\documentclass[a4paper,12pt]{article}

\usepackage[top=3.0cm,bottom=3.0cm,left=2.25cm,right=2.25cm]{geometry}
\usepackage{amsmath,amsthm,mathrsfs,graphicx,amsfonts}
\usepackage{bm}
\usepackage{cases}
\usepackage[english]{babel}
\usepackage{amsmath,amsthm}
\usepackage{amsfonts}
\usepackage{latexsym}
\usepackage{graphicx}
\usepackage{txfonts}
\usepackage[numbers,sort&compress]{natbib}
\usepackage[natural]{xcolor}
\usepackage{rotating}
\usepackage{mathtools}
\usepackage{enumitem}

\newtheorem{lem}{Lemma}[section]
\newtheorem{thm}[lem]{Theorem}

\newtheorem{prob}[lem]{Problem}
\newtheorem{claim}[lem]{\indent Claim}
\newtheorem{conj}[lem]{Conjecture}
\begin{document}
\title{Equating three degrees of graphs}
\date{}
\author{Zhen Liu\footnote{Email: 1552580575@qq.com}, ~~Qinghou Zeng\footnote{Research partially supported by National Key R\&D Program of China (Grant No. 2023YFA1010202) and National Natural Science Foundation of China (Grant No. 12371342). Email: zengqh@fzu.edu.cn (Corresponding author)}\\
{\small Center for Discrete Mathematics, Fuzhou University, Fujian, 350003, China}}

\maketitle

\maketitle
\begin{abstract}
In this paper, we prove that, for every graph with at least 5 vertices, one can delete at most 3 vertices such that the subgraph obtained has at least three vertices with the same degree. This solves an open problem of Caro, Shapira and Yuster [Electron. J. Combin. 21 (2014) P1.24].
\end{abstract}
\section{Introduction}\label{Intro}
All graphs considered here are finite, simple and undirected, unless otherwise specified.
For a graph $G$, the \emph{repetition number}, denoted by $\text{rep}(G)$, is the maximum multiplicity of a vertex degree in $G$. Trivially, $\text{rep}(G)\ge2$ for any graph $G$ with at least two vertices. There are also simple constructions showing that the equality holds for infinitely many graphs. Repetition numbers of graphs and hypergraphs have been widely studied by various researchers, see \cite{BBL, Bol, BS1997, CHP, CHPSW, CW2009, EFR}.

Since there are infinitely many graphs having repetition number two, Caro, Shapira and Yuster \cite{CSY} asked that  what is the smallest number of vertices one needs to delete from a graph in order to increase the repetition number of the remaining induced subgraph to 3 or higher. For any positive integer $k$, let $C(k)$ denote the least integer such that any $n$-vertex graph has an induced subgraph with at least $n-C(k)$ vertices whose repetition number is at least $\min\{k, n-C(k)\}$. Trivially, $C(1)=C(2)=0$. Caro, Shapira and Yuster  \cite{CSY} established that $\Omega(k\log k)\le C(k)\le (8k)^k$ for any $k\ge3$.
Specially, for the first nontrivial case $C(3)$, the authors \cite{CSY} proved that $C(3)\le6$, and the exact value was left as an open problem.
\begin{prob}[Caro, Shapira and Yuster \cite{CSY}]\label{Prob}
Determine the exact value of $C(3)$.
\end{prob}

\begin{figure}
    \centering
    \includegraphics[width=0.25\linewidth]{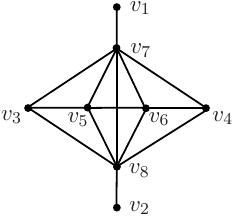}
    \caption{$C(3)\geq 3$}
    \label{fig}
\end{figure}
Recently, the upper bound of $C(3)$ was continuously improved by Kogan \cite{Kog}, and by Sun, Hou and Zeng \cite{SHZ}. In this paper, we consider Problem \ref{Prob} and determine that $C(3)=3$. We mention that the lower bound $C(3)\ge3$ (see Figure \ref{fig}) is constructed by Caro, Shapira and Yuster \cite{CSY}. In addition, the path on 4 vertices shows that $|V(G)|\ge5$ is necessary in the following theorem.

\begin{thm}\label{Delete3}
For any graph $G$ with at least $5$ vertices, one can delete at most $3$ vertices such that the subgraph obtained has at least three vertices with the same degree. Consequently, $C(3)=3$.
\end{thm}

\noindent\textbf{Notation}. Let $G$ be a graph. For any $v\in V(G)$, denote $N_{G}(v)$ the set of \emph{neighbors} of $v$ in $G$ and $d_G(v)$ the \emph{degree} of $v$ in $G$. Write $N_{G}[v]=N_{G}(v)\cup\{v\}$. For any $S\subseteq V(G)$, let $G[S]$ denote the induced subgraph of $G$ on $S$.  For each $S, T\subseteq V(G)$, let $e_G(S,T)$ denote the number of edges of $G$ with one end in $S$ and the other end in $T$.  We will drop the reference to $G$ when there is no danger of confusion. In addition, we usually write $[k]:=\{1,\ldots,k\}$ for any integer $k\ge2$.

\section{Feasible sets and related lemmas}\label{Feasible}
Let $G$ be a graph, and let $S=\{x,y,z\}$ be a subset of $V(G)$ such that $d(x)\leq d(y)\leq d(z)$. We follow the definition of a feasible set in \cite{SHZ}.  Call $S$ \emph{balanceable} if one of the following conditions holds:
\begin{itemize}
\item[(C1)] $G[S]$ is an independent set;
\item[(C2)] $G[S]$ is a clique;
\item[(C3)] $G[S]$ contains only the edge $xy$;
\item[(C4)] $G[S]$ contains only the edges $xy$ and $xz$.
\end{itemize}
Furthermore, call $S$ \emph{accessible} if one of the following conditions holds:
\begin{itemize}
\item[(C5)] $G[S]$ contains only the edge $xz$, and $N(y)\backslash N(z)\neq\emptyset$;
\item[(C6)] $G[S]$ contains only the edges $xy$ and $yz$, and $N(x)\backslash N[y]\neq\emptyset$;
\item[(C7)] $G[S]$ contains only the edge $yz$, and $N(x)\backslash N(y)\neq\emptyset$, $N(x)\backslash N(z)\neq\emptyset$;
\item[(C8)] $G[S]$ contains only the edges $xz$ and $yz$, and $N(x)\backslash N[z]\neq\emptyset$, $N(y)\backslash N[z]\neq\emptyset$.
\end{itemize}
Usually, we call $S$ \emph{feasible} for short if it is balanceable or accessible.

We first present three lemmas given by Sun, Hou and Zeng \cite{SHZ}, which are useful in our proof. For any set $S=\{x,y,z\}$ with $d(x)\leq d(y)\leq d(z)$, let $p(S):=d(z)-d(y)$ and $q(S):=d(y)-d(x)$.

\begin{lem} [Sun, Hou and Zeng \cite{SHZ}]\label{pqmaxpq}
If $G$ contains a feasible set $S$, then one can delete at most
$$p(S)+q(S)+\max\{p(S),q(S)\}$$
vertices from $G$ such that the subgraph obtained has at least three vertices with the same degree.
\end{lem}

\begin{lem} [Sun, Hou and Zeng \cite{SHZ}]\label{P4}
Let $X=\{v_1, v_2, v_3, v_4\}$ be a subset of $G$ such that $d(v_1)\leq d(v_2)\leq d(v_3)\leq d(v_4)$. If $X$ doesn't contain balanceable sets, then $G[X]$ is an induced path such that $v_1$ and $v_2$ are terminal vertices.
\end{lem}

\begin{lem} [Sun, Hou and Zeng \cite{SHZ}]\label{u3}
Let $G$ be a graph and $U=\{u_1,u_2,u_3,u_4,u_5\}\subseteq V(G)$. If $d(u_1)\leq \ldots\leq d(u_5)$, then there is a feasible set $S\subseteq U$ such that $u_3\in S$.
\end{lem}

The following lemma plays a key role in our proof of Theorem \ref{Delete3}.
\begin{lem}\label{Main-Lem}
Let $G$ be a graph and let $U=\{v_1,v_2,v_3,v_4\}\subseteq V(G)$ with $d(v_1)=d(v_2)=d$ and $d(v_3)=d(v_4)=d+2$ for any integer $d\ge1$. If $U$ contains a balanceable set, then one can delete at most $3$ vertices from $G$ such that the subgraph obtained has at least three vertices with the same degree.
\end{lem}
\begin{proof}[\bf {Proof}]
Suppose that $S\subseteq U$ is a balanceable set. Clearly, the degrees of the three vertices in $S$ are either $d, d, d+2$ or $d, d+2, d+2$. We may assume that the three vertices in $S$ has degrees $d, d, d+2$, respectively; otherwise, we may take the complement of $G$ as our graph. Without loss of generality, let $S=\{v_1, v_2, v_3\}$.
\begin{claim}\label{312}
The set $N(v_3)\backslash (N[v_1] \cup N[v_2])$ is empty.
\end{claim}
\begin{proof}
Suppose that there exists a vertex $v\in N(v_3)\backslash (N[v_1] \cup N[v_2])$. Clearly, the deleting of $v$ reduces the degrees of $v_1, v_2, v_3$ from $(d, d, d+2)$ to $(d, d, d+1)$. Note that $\{v_1,v_2,v_3\}$ remains balanceable in the resulting graph as it is balanceable in $G$. It follows from Lemma \ref{pqmaxpq} that we can delete at most two more vertices to equate the degrees of $v_1, v_2, v_3$. Overall, we delete at most 3 vertices, as desired.
\end{proof}
\begin{claim}\label{124}
The set $\{v_1, v_2, v_4\}$ is feasible.
\end{claim}
\begin{proof}
Suppose that $\{v_1, v_2, v_4\}$ is not a feasible set. It follows from the conditions (C5)-(C8) that at least one of $N(v_1)\setminus N[v_2]=\emptyset$, $N(v_1)\setminus N[v_4]=\emptyset$ and $N(v_2)\setminus N[v_4]=\emptyset$ holds.

If $N(v_1)\setminus N[v_2]=\emptyset$, then $|N[v_1] \cup N[v_2]|\le d+2$ as $d(v_1)=d(v_2)=d$. Note that $d(v_3)=d+2$ and $N(v_3)\setminus(N[v_1]\cup N[v_2])=\emptyset$ by Claim \ref{312}. This forces that $N[v_1]\cup N[v_2]=N(v_3)$, meaning that $v_3\in N[v_1]\cup N[v_2]$. This leads to a contradiction as $v_3\notin N(v_3)$.

If $N(v_1)\setminus N[v_4]=\emptyset$, then we first assert that
\begin{align}\label{v12}
v_1v_2\notin E(G).
\end{align}
Suppose that $v_1v_2\in E(G)$. It follows from the conditions (C2)-(C4) that $v_1v_4\notin E(G)$ and $v_2v_4\in E(G)$ as $\{v_1, v_2, v_4\}$ is not feasible. This further implies that $N(v_1)\setminus N[v_2]=\emptyset$ by the condition (C6). It follows that $N(v_3)\subseteq N[v_2]$ by Claim \ref{312}. This leads to a contradiction as $d(v_3)=d(v_2)+2$. Thus, we know that \eqref{v12} holds. Recall that $\{v_1, v_2, v_3\}$ is a balanceable set. This together with \eqref{v12} yields that $\{v_1, v_2, v_3\}$ is actually an independent set, i.e., $v_i\notin N(v_3)$ for each $i\in[2]$. Note that $N(v_3)\setminus(N[v_1]\cup N[v_2])=\emptyset$ by Claim \ref{312} and $d(v_1)=d(v_2)=d(v_3)-2$. Thus, for each $i\in[2]$
\begin{align}\label{v3i}
|N(v_3)\setminus N[v_i]|=|N(v_3)\setminus N(v_i)|\ge2 \;\;\text{and}\;\; N(v_3)\setminus N(v_i)\subseteq N(v_{3-i})\setminus N(v_i).
\end{align}
If $v_2v_4\in E(G)$, then it follows from \eqref{v3i} that $N(v_3)\setminus N(v_2)\subseteq N(v_1)\setminus N(v_2)\subseteq N(v_4)$ as $N(v_1)\setminus N[v_4]=\emptyset$ and $v_4\in N(v_2)$. If $v_2v_4\notin E(G)$, then it follows from the condition (C5) that $N(v_2)\setminus N[v_4]=\emptyset$ as $\{v_1, v_2, v_4\}$ is not feasible. This together with \eqref{v3i} further implies that $N(v_3)\setminus N(v_1)\subseteq N(v_2)\setminus N(v_1)\subseteq N(v_4)$ as $v_4\notin N(v_2)$. Thus, we conclude that there exists some $j\in[2]$ such that
\[
N(v_3)\setminus N(v_j)\subseteq N(v_4).
\]
Recall that $|N(v_3)\setminus N(v_i)|\ge2$ for each $i\in[2]$ by \eqref{v3i}. Consequently, we can always delete two vertices in $N(v_3)\setminus N(v_j)$ to equate the degrees of $v_{j}$, $v_3$ and $v_4$ from $(d, d+2, d+2)$ to $(d, d, d)$ for some $j\in[2]$, as required.

If $N(v_2)\setminus N[v_4]=\emptyset$, then an argument similar to that used in the case of $N(v_1)\setminus N[v_4]=\emptyset$ implies the desired result. Thus, we complete the proof of Claim \ref{124}.
\end{proof}

By Claim \ref{124}, we know that $\{v_1, v_2, v_4\}$ is a feasible set. An argument similar to that used in Claim \ref{312} implies that
\begin{align}\label{412}
N(v_4)\setminus (N[v_1] \cup N[v_2])=\emptyset,
\end{align}
where the only different is that $\{v_1,v_2,v_4\}$ remains feasible in the resulting graph if we delete some vertex in $N(v_4)\backslash (N[v_1] \cup N[v_2])$ by checking the conditions (C1)-(C8) one by one. This together with Claim \ref{312} yields that $(N(v_3)\cup N(v_4))\subseteq(N[v_1] \cup N[v_2])$. Recall that $U=\{v_1,v_2,v_3,v_4\}$. Let
\[
W:=\big((N(v_1)\setminus N(v_2))\cup (N(v_2)\setminus N(v_1))\big)\setminus U.
\] In what follows, we show that
\begin{align}\label{W4}
e\left(\{v_3,v_4\},W\right)-|W|\ge4,
\end{align}
meaning that there exists two vertices $u_1$ and $u_2$ in $N(v_j)\setminus (N(v_{3-j})\cup U)$ for some $j\in[2]$ such that $u_iv_3\in E(G)$ and $u_iv_4\in E(G)$ for each $i\in[2]$. It follows that we can always equate the degrees of $v_{3-j}$, $v_3$ and $v_4$ from $(d, d+2, d+2)$ to $(d, d, d)$ by deleting $u_1$ and $u_2$ for some $j\in[2]$, as required. Now, we prove \eqref{W4}. For two vertices $u,v\in V(G)$, let $\textbf{1}_{uv}=1$ if $uv\in E(G)$ and $\textbf{1}_{uv}=0$ if $uv\notin E(G)$.
Note that
\[
|W|=d(v_1)+d(v_2)-2|\left(N(v_1)\cap N(v_2)\right)\setminus U|-2\cdot\textbf{1}_{v_1v_2}-e\left(\{v_1,v_2\},\{v_3,v_4\}\right)
\]
and
\begin{align*}
e(\{v_3,v_4\},W)\ge d(v_3)+d(v_4)-2\cdot\textbf{1}_{v_3v_4}-e\left(\{v_3,v_4\},\{v_1,v_2\}\right)-2|\left(N(v_1)\cap N(v_2)\right)\setminus U|.
\end{align*}
Thus, we have
\[e
(\{v_3,v_4\},W)-|W|\ge d(v_3)+d(v_4)-d(v_1)-d(v_2)+2(\textbf{1}_{v_1v_2}-\textbf{1}_{v_3v_4})=4+2(\textbf{1}_{v_1v_2}-\textbf{1}_{v_3v_4}).
\]
This implies \eqref{W4} unless $v_1v_2\notin E(G)$ and $v_3v_4\in E(G)$. However, $v_1v_2\notin E(G)$ means that $\{v_1, v_2, v_3\}$ is an independent set as $\{v_1, v_2, v_3\}$ is a balanceable set, i.e., $v_3\notin (N[v_1] \cup N[v_2])$. Note that $N(v_4)\subseteq(N[v_1] \cup N[v_2])$ by \eqref{412}. This leads to a contradiction as $v_3\in N(v_4)$. Thus, the inequality \eqref{W4} always holds. This finally completes the proof of Lemma \ref{Main-Lem}.
\end{proof}
\section{Proof of Theorem \ref{Delete3}}
In this section, we prove Theorem \ref{Delete3}.
\begin{proof} [{\bf Proof of Theorem \ref{Delete3}}]
For an integer $n\ge5$, let $G$ be an $n$-vertex graph with $\text{rep}(G)=2$. First, we may assume that $G$ contains no isolated vertices. Otherwise, it is easy to check that Theorem \ref{Delete3} holds for any $G$ with $n\le6$. It suffices to prove Theorem \ref{Delete3} for $n\ge7$. Thus, we may take the subgraph $G'$ of $G$ by deleting the isolated vertices as our new graph. Clearly, $G'$ has at least 5 vertices, repetition number exactly two and positive degrees.

Suppose that one must delete more than 3 vertices from $G$ to obtain an induced subgraph with at least three vertices with the same degree. Let
\[
S=\left\{d\in[n-1]: \text{there exist exactly two vertices in $G$ with degree $d$}\right\}
\]
and
\[
T=\left\{d\in[n-1]: \text{there exists no vertex in $G$ with degree $d$}\right\}.
\]
Now, we proceed our proof by showing the following series of claims. For some $x\notin[n-1]$, we use $x\in T$ (or $x\in S$) to denote that there exists $y\in T$  (or $y\in S$) such that $x\equiv y$ (mod $n-1$).

\begin{claim}\label{0}
If $d\in S$, then either $d-1\in T$ or $d+1\in T$.
\end{claim}
\begin{proof}
We first prove this for $2\le d\le n-2$. Suppose that there exist 4 vertices $v_1,v_2,v_3,v_4$ in $G$ with degrees $d-1$, $d$, $d$, $d+1$, respectively. If $X:=\{v_1,v_2,v_3,v_4\}$ contains a balanceable set $S_0$, then we can delete at most 3 vertices to equate the degrees of the vertices in $S_0$ by Lemma \ref{pqmaxpq}, a contradiction. Otherwise, $G[X]$ is an induced path with $v_1,v_2$ as its terminal vertices by Lemma \ref{P4}. Clearly, this path has two possible degree sequences $(d(v_1),d(v_3),d(v_4),d(v_2))=(d-1,d,d+1,d)$ and $(d(v_1),d(v_4),d(v_3),d(v_2))=(d-1,d+1,d,d)$. For the first case, the deleting of $v_4$ reduces the degrees of $v_1,v_2,v_3$ from $(d-1, d, d)$ to $(d-1, d-1, d-1)$, a contradiction; for the second case, the deleting of $v_1$ reduces the degrees of $v_2,v_3,v_4$ from $(d, d, d+1)$ to $(d, d, d)$, a contradiction.

If $d=1\in S$, then we show that $n-1\in T$ if $2\notin T$. Suppose that $n-1\notin T$. This means that there exist two vertices $u,v\in V(G)$ with $d(u)=n-1$ and $d(v)=2$. Let $w\in N(v)\setminus \{u\}$ and $d(v_i)=1$ for each $i\in[2]$. Clearly, $v,v_1,v_2$ have the same degree 1 in the resulting graph by deleting $w$, a contradiction. If $d=n-1\in S$, then every vertex in $G$ has degree at least 2, i.e., $d+1\in T$.
\end{proof}

\begin{claim}\label{1}
If $\{d, d+1\}\subseteq S$, then $\{d-1,d+2\}\subseteq T$.
\end{claim}
\begin{proof}
Note that $d\in S$ and $d+1\notin T$. This implies that $d-1\in T$ by Claim \ref{0}. Similarly, we have $d+2\in T$ as $d+1\in S$ and $d\notin T$.
\end{proof}

\begin{claim}\label{2}
If $\{d, d+2\}\subseteq S$, then $\{d-1,d+1,d+3\}\subseteq T$.
\end{claim}
\begin{proof}
We first prove this for $2\le d\le n-4$. Let $U=\{v_1,v_2,v_3,v_4\}\subseteq V(G)$ with $d(v_1)=d(v_2)=d$ and $d(v_3)=d(v_4)=d+2$. If there exists a vertex $u\in V(G)$ with $d(u)=d+1$, then $U\cup\{u\}$ must contain a feasible set $S_0$ such that $u\in S_0$ in view of Lemma \ref{u3}. It follows from Lemma \ref{pqmaxpq} that we can always delete at most 3 vertices to equate the degrees of the vertices in $S_0$, a contradiction.
By Lemma \ref{Main-Lem}, $U$ doesn't contain balanceable sets, implying that $G[U]$ is an induced path with $v_1,v_2$ as its terminal vertices. Hence, $v_1v_2\notin E(G)$ and $v_3v_4\in E(G)$. If there exists $v_0\in V(G)$ with $d(v_0)=d-1$, then $\{v_0,v_1,v_2\}$ is feasible as $v_1v_2\notin E(G)$ and $d(v_1)=d(v_2)$; if there exists $v_5\in V(G)$ with $d(v_5)=d+3$, then $\{v_3,v_4,v_5\}$ is also feasible as $v_3v_4\in E(G)$ and $d(v_3)=d(v_4)$. In both cases, we could delete at most 2 vertices to make their degrees equal by Lemma \ref{pqmaxpq}, a contradiction.

If $d=1\in S$, then it suffices to check that $n-1\in T$ as the other cases are still valid by the proof above. However, we may also set $v_5\in V(G)$ with $d(v_5)=n-1$ for the case $n-1\in T$ and the desired result follows in a similar way as before.

If $d=n-3\in S$, then we shall only need to check that $1\in T$. This is clearly true as every vertex in $G$ has degree at least 2 on account of $n-1\in S$.

If $d=n-2\in S$, then we have $1\in S$ and we should prove that $\{n-3,n-1,2\}\subseteq T$. Let $U=\{v_1,v_2,v_3,v_4\}\subseteq V(G)$ with $d(v_1)=d(v_2)=1$ and $d(v_3)=d(v_4)=n-2$. Clearly, $G[U]$ is an induced path with $v_1,v_2$ as its terminal vertices, and $(\{v_3,v_4\},V(G)\setminus U)$ is a complete bipartite subgraph. If there exists a vertex $v_0\in V(G)$ with $d(v_0)=2$, then $v_0,v_1,v_2$ have the same degree 0 in the resulting graph by deleting $v_3$ and $v_4$, a contradiction. If there exists a vertex $v_5\in V(G)$ with $d(v_5)=n-3$ or $d(v_5)=n-1$, then by Lemma \ref{pqmaxpq} we can always delete at most 3 vertices to equate the degrees of $v_3,v_4,v_5$ as they form a triangle, a contradiction.

If $d=n-1\in S$, then we have $2\in S$ and we should prove that $\{n-2,1,3\}\subseteq T$. Clearly, $1\in T$.  Let $U=\{v_1,v_2,v_3,v_4\}\subseteq V(G)$ with $d(v_1)=d(v_2)=2$ and $d(v_3)=d(v_4)=n-1$. It is easy to see that $v_3v_4\in E(G)$ and $(\{v_3,v_4\},V(G)\setminus \{v_3,v_4\})$ is a complete bipartite subgraph. If there exists a vertex $v_0\in V(G)$ with $d(v_0)=3$, then $v_0,v_1,v_2$ have the same degree 2 in the resulting graph by deleting some vertex $w\in N(v_0)\setminus \{v_3,v_4\}$, a contradiction. If there exists a vertex $v_5\in V(G)$ with $d(v_5)=n-2$, then by Lemma \ref{pqmaxpq} we can always delete at most 3 vertices to equate the degrees of $v_3,v_4,v_5$ as they form a triangle, a contradiction.
\end{proof}

Let $S_1=\dot{\cup}_{i=1}^\ell S_{1i}$, where each $S_{1i}$ is a maximal set of the form $\{d,d+2,\ldots,d+2j\}\subseteq S$ for some $d\in S$ and $j\ge1$. Subject to this, we may further assume that $\ell$ is maximal. Let $T_{1i}=\{d-1,d+1,\ldots,d+2j+1\}$ for each $S_{1i}$. Clearly, $T_{1i}\cap T_{1j}=\emptyset$ for $i\neq j\in[\ell]$ by the maximality of $S_{1i}$, and $T_1=\dot{\cup}_{i=1}^\ell T_{1i}\subseteq T$ by Claim \ref{2}. Thus, we conclude that $|T_1|\ge|S_1|$. Set $S_2=S\setminus S_1$ and $T_2=T\setminus T_1$. By the maximality of each $S_{1i}$, we have $T_1\cap T_2=\emptyset$. Moreover, for any two pairs $\{d_1, d_1+1\}\subseteq S_2$ and $\{d_2, d_2+1\}\subseteq S_2$, we have $\{d_1-1, d_1+2\}\cap\{d_2-1, d_2+2\}=\emptyset$ by the maximality of $\ell$. Consequently, we have $|T_2|\ge|S_2|$ by Claims \ref{0} and \ref{1}. This finally implies that $|T|\ge|S|$. Note also that
\[
n-1=|S|+(n-2|S|)+|T|,
\]
i.e. $|T|=|S|-1$. This leads to a contradiction and completes the proof of Theorem \ref{Delete3}.
\end{proof}

\section{Concluding remarks}
In this paper, we give the exact value of $C(3)$ for general graphs. It is still open to determine $C(k)$ for any integer $k\ge4$, and the case analysis becomes signiﬁcantly involved when trying to adapt our arguments even for $C(4)$.
There is also an analogous problem about equating maximum degrees initially studied by Caro and Yuster \cite{CY2010}. For any integer $k\ge2$, let $f_k(G)$ denote the minimum number of vertices that we need to delete from $G$ such that the remaining induced subgraph has its maximum degree shared by at least $k$ vertices. Denote by $f_k(n)$ the maximum of $f_k(G)$ as $G$ ranges over all graphs with $n$ vertices. Caro and Yuster \cite{CY2010} posed the following conjecture.
\begin{conj}[Caro and Yuster \cite{CY2010}]
For any fixed integer $k\geq2$, $f_k(n)=\Theta(\sqrt{n})$.
\end{conj}
The authors solved this conjecture for $k=2, 3$, and the cases $k\geq4$ remain widely open. Particularly, we mention that the exact value of $f_2(n)$ was solved by Caro, Lauri and Zarb \cite{CLZ2018}. However, the problem for the exact value of $f_k(n)$ is widely open for any $k\geq3$, and is still active.  We refer the reader to \cite{CLZ2019,CY2020, FGJ, GP2018} for more problems and results in this area.

\end{document}